\documentclass[11pt]{amsart}
\usepackage{amssymb,latexsym}

\def \R {\Bbb R}

\def\div{{\rm div}}

\def \vs{\vspace*{0.1cm}}

\def \ds{\displaystyle}

\def\triangle{\Delta}

\def\D{\Delta}

\def \ds{\displaystyle}

\def\be1{{\begin{equation}}}
\def\ee1{{\end{equation}}}

\def\D{\Delta}

\def\C{\mathbb C}

\def\S{{\mathbb S}}

\def\part{\partial}
\def\R{{\mathbb R}}

\def\ba{\begin{array}}
\def\ea{\end{array}}

\numberwithin{equation}{section}

\newtheorem{lemma}{Lemma}[section]
\newtheorem{proposition}[lemma]{Proposition}
\newtheorem{theorem}[lemma]{Theorem}

\title{A conformal integral invariant on Riemannian foliations}

\author{Guofang Wang}
\address{ Albert-Ludwigs-Universit\"at Freiburg,
Mathematisches Institut,
Eckerstr. 1,
79104 Freiburg, Germany}
\email{guofang.wang@math.uni-freiburg.de}
\thanks{The project is supported by SFB/TR71 ``Geometric partial differential equations'' of DFG}
\author{Yongbing Zhang}\address{University of Science and Technology of China, Department of Mathematics,
           230026 Hefei, China 
           \\
\and Albert-Ludwigs-Universit\"at Freiburg, Mathematisches Institut,  Eckerstr. 1, 79104 Freiburg, Germany}
\email{ybzhang@amss.ac.cn}

\date{}

\begin{document}

\begin{abstract}
Let $M$ be a closed manifold which admits a foliation structure $\mathcal{F}$ of codimension $q\geq 2$ and a bundle-like metric $g_0$.
Let $[g_0]_B$ be the space of bundle-like metrics which differ from $g_0$
only along the horizontal directions by a multiple of a positive basic function.
Assume $Y$ is a transverse conformal vector field and the mean curvature of the leaves of $(M,\mathcal{F},g_0)$ vanishes.
We show that the integral $\int_MY(R^T_{g^T})d\mu_g$ is independent of the choice of $g\in [g_0]_B$, where $g^T$ is the transverse metric induced by $g$
and $R^T$ is the transverse scalar curvature.
Moreover if $q\geq 3$, we have $\int_MY(R^T_{g^T})d\mu_g=0$ for any $g\in [g_0]_B$.
However there exist codimension $2$ minimal Riemannian foliations $(M,\mathcal{F},g)$ and transverse conformal vector fields $Y$
such that $\int_MY(R^T_{g^T})d\mu_g\neq 0$. Therefore, it is a nontrivial obstruction for
the transverse Yamabe problem on minimal Riemannian foliation of codimension $2$.
\end{abstract}

\maketitle

\section{Introduction}

In \cite{Kazdan-Warner} Kazdan and Warner discovered an obstruction to
the existence of metrics with prescribed scalar curvature on $\S^2$.
Let $(\S^2, g_{\S^2})$ be the unit sphere in ${\R}^3$ with the standard metric and $h$  a given function on $\S^2$.
Kazdan and Warner found that if $\varphi$ a solution to the equation
\begin{equation} \label{prescribedscalar}
\D_{g_{\S^2}}\varphi+2-he^{-\varphi}=0,
\end{equation}
then for any first order spherical harmonic $F$
(i.e. the restriction to $\S^2$ of a linear function in ${\R}^3$) it holds that
\begin{equation}\label{integrabilitycondition}
\int_{\S^2} g(\nabla F,\nabla h)e^{-\varphi}d\mu_{g_{\S^2}}=0.
\end{equation}
If $g=e^{-\varphi}g_{\S^2}$ and $\varphi$ satisfies (\ref{prescribedscalar}),
the scalar curvature of ${g}$ is equal to $h$. 
Hence, (\ref{integrabilitycondition}) is just
\begin{equation}\label{obstruction1}
\int_{\S^2}(\nabla F)(R_{{g}})d\mu_{{g}}=0.
\end{equation}

Similar integrability condition as (\ref{integrabilitycondition}) was proved for higher dimensional spheres \cite{Kazdan-Warner2}.
The integrability condition (\ref{integrabilitycondition}) was generalized to any closed Riemannian manifold with a conformal vector field,
see \cite{Bourguignon,BourguignonEzin}.
Note that $\nabla F$ in (\ref{integrabilitycondition}) is a conformal vector field with respect to the standard metric on $\S^2$.
Let $(M,g_0)$ be a closed Riemannian manifold and $[g]$ be the conformal class of $g_0$.
Bourguignon and Ezin proved that for any conformal vector field $Y$
\begin{equation}\label{BouguignonEzin}
\int_{M}Y(R_{{g}})d\mu_{{g}}=0, \quad \forall \, {g}\in [g_0].
\end{equation}
By (\ref{BouguignonEzin}), they found new functions which cannot be realized as the scalar curvature of a Riemmanian metric on $\S^2$.
Remark that the proof given in \cite{BourguignonEzin} is different for dimension $n\ge 3$ and for dimension $n=2$.
There is another interesting proof given in \cite{Bourguignon}. See also \cite{Futaki2}. One  first shows that
\begin{equation}\label{int1}
\int_{M}Y(R_{{g}})d\mu_{{g}}
\end{equation}
is a conformal invariant, and then shows that this invariant vanishes by a result of Obata.
For the case that $(M,g)$ is a compact manifold with boundary, 
see \cite{Schoen}.

Before \cite{Bourguignon} and \cite{BourguignonEzin},
Futaki \cite{Futaki} found an analogous invariant for the complex Monge-Ampere equation
on K\"ahler-Einstein manifolds of positive first Chern class.
This is the well-known Futaki invariant, which 
is one of the main obstructions to the existence of K\"ahler-Einstein metrics of positive first Chern class.

Very recently we have studied a Yamabe type problem on Riemannian foliations, i.e., finding a
bundle-like metric $g$ in a given  basic conformal class with  constant
transverse scalar curvature. This is a natural generalization of the
Yamabe problem to Riemannian foliation.

Let us first briefly recall the definitions
of basic conformal class.
Let $(M,\mathcal{F},g_0)$ be a closed Riemannian foliation of codimension $q\geq 2$ 
with a bundle-like metric $g_0$.
From now on we assume that $M$ is oriented and $\mathcal{F}$ is transversally oriented.
Let $\mathcal{L}$ denote the integrable subbundle given by $\mathcal{F}$.
The bundle-like metric $g_0$ induces a holonomy invariant transverse metric $g_0^T$
on the normal bundle $\nu(\mathcal{F})=TM/ \mathcal{L}$ of the foliation.
Let $\Omega_B^0(M,\mathcal{F})$ denote the space of all basic functions and
\begin{equation}\label{basicconformalclass}
[g_0]_B=\{g=g_0|_{\mathcal{L}}+e^u g_0|_{\mathcal{L}^\perp}: u\in \Omega_B^0(M,\mathcal{F}) \}.
\end{equation}
$[g_0]_B$ is called the basic conformal class of the bundle-like metric $g_0$.
Any transverse metric of the form $g^T=e^u g_0^T$ is holonomy invariant if and only if
$u$ is a basic function.
We call a transverse metric $g^T$ conformal to $g_0^T$ if $g^T=e^u g_0^T$ for a basic function $u$.
We denote the space of all conformal transverse metrics of $g_0^T$ by $[g_0^T]_B$.
There is a one to one correspondence between $[g_0]_B$ and $[g_0^T]_B$
and we denote by $g^T$ the transverse metric induced by $g\in [g_0]_B$ .
We denote by $R^T_{g^T}$ the transverse scalar curvature of $g^T$.
A Riemannian foliation is called minimal if the mean curvature of the leaves vanishes.
For geometry of foliations, see for instance \cite{Tondeur} or Section 2 below.
For the definition of transverse vector field, see \cite{JungJung} or Section 2 below.
First, we observe that a similar integral  like (\ref{int1}) is invariant in a basic conformal class.

\begin{theorem}\label{mainthm}
Let $Y$ be a transverse conformal vector field on the minimal Riemannian foliation $(M,\mathcal{F},g_0)$.
Then the integral
\begin{equation}\label{int2}
\int_MY(R^T_{g^T})d\mu_g.
\end{equation}
 is independent of the choice of $g\in [g_0]_B$.
 \end{theorem}

Without the assumption that the Riemannian foliation $(M,\mathcal{F},g_0)$ is minimal,
$\int_MY(R^T_{g^T})d\mu_g$ may depend on $g\in [g_0]_B$.

It is easy to see that  invariant (\ref{int2}) is an obstruction of the following transverse Yamabe problem.

 \

 \noindent{\it Transverse Yamabe problem.}
 Let $(M,\mathcal{F},g_0)$ be a Riemannian foliation with a bundle-like metric $g_0$. Does there exist
 any $g^T\in [g_0^T]_B$ such that
 \[R^T_{g^T}=const.?\]

 \

 This is a natural generalization of the ordinary Yamabe problem,
 which was resolved by Yamabe, Trudinger, Aubin and finally by Schoen.
 The resolution of the Yamabe problem is a milestone of geometric analysis.
 An equivariant version of the Yamabe problem has been studied by Hebey-Vaugon \cite{HV}.
 Very recently, a fully nonlinear Yamabe type problem has been studied by Viaclovsky, Chang-Gursky-Yang
 and many other mathematicians. See a survey of Viaclovsky \cite{V0}.
 For the further study of the ordinary Yamabe problem, see a survey of Brendle and Marques \cite{BM}.

 It is clear that the invariant (\ref{int2}) is an obstruction of the transverse Yamabe problem,
 at least for minimal foliations, which is the most interesting case: if there is a solution then invariant (\ref{int2}) must vanish.
 Unlike the ordinary Yamabe problem, now invariant (\ref{int2}) is not a trivial invariant.

  \begin{theorem}\label{mainthm2} There are examples of minimal Riemannian
  foliation of codimension $q=2$ with a transverse conformal vector field $Y$ such that
 invariant (\ref{int2}) does not vanish. Hence on such a Riemannian foliation there is
 no solution for the transverse Yamabe problem.
\end{theorem}

The examples come from our study of 3-dimensional Sasaki-Ricci flow in \cite{Wang}.
For the Sasaki-Ricci flow see \cite{SWZ}.

However, for higher codimension ($q>2$), this invariant still vanishes, though we believe that in general
there exist obstructions for the transverse Yamabe problem.

When the leaves of a Riemannian foliation $(M,\mathcal{F},g_0)$ are all compact,
the leaf space with the induced metric from $g_0$ is a Riemannian orbifold.
Moreover, its scalar curvature is exactly the transverse scalar curvature. Hence,
in this case the transverse Yamabe problem is equivalent to the Yamabe problem on orbifolds,
which has been studied by Akutagawa-Botvinnik in \cite{AB1} and \cite{A}.
Very recently, Viaclovsky \cite{V1} gave interesting examples of $4$-dimensional orbifold,
on which the orbifold Yamabe problem has no solution.
It is an interesting and natural  question if one can find a similar obstruction for the orbifold Yamabe problem.
Our results can only provide an obstruction for 2-dimensional orbifolds.
The non-existence of solutions to the orbifold Yamabe problem on (bad) 2-dimensional orbifolds
follows from the result of Langfang Wu in \cite{Wu}.


The paper is organized as follows. In Section 2, we first provide preliminaries on Riemannian foliation
and the transverse Yamabe problem. Then we show the conformal invariance of (\ref{int2}).
In Section 3 we provide our examples of minimal Riemannian foliation of codimension 2 with nonvanishing (\ref{int2}).

We would like to thank Professor A. Futaki for his interest and telling us the reference \cite{Futaki2}.

\section{A conformal integral invariant}

In this section we show that on a minimal foliation $(M,\mathcal{F},g_0)$ with a transverse conformal vector field $Y$,
the integral (\ref{int2}) is independent of the choice of $g\in [g_0]_B$.
If in addition the codimension of the foliation is greater or equal to $3$,
we show that (\ref{int2}) is equal to zero for $g\in [g_0]_B$.

Let $(M,\mathcal{F})$ be a closed manifold with a foliation $\mathcal{F}$ of codimension $q$.
Let $\mathcal{L}$ denote the integrable subbundle given by $\mathcal{F}$ and $\nu(\mathcal{F})=TM/ \mathcal{L}$.
We denote the quotient map by $\pi: T M\rightarrow \nu(\mathcal{F})$.
Any Riemannian metric $g$ on $M$ provides a splitting of the exact sequence of bundles
$$0\rightarrow \mathcal{L}\rightarrow TM\rightarrow \nu(\mathcal{F})\rightarrow 0$$
and an isomorphism of bundles
$\sigma: \nu(\mathcal{F}) \rightarrow  \mathcal{L}^\perp $ satisfying $\pi\circ\sigma =id$.
The transverse metric $g^T$ corresponding to $(M,\mathcal{F},g)$ is defined by
$$g^T(s_1,s_2)=g(\sigma s_1,\sigma s_2), \quad s_1, s_2\in \Gamma \nu(\mathcal{F}).$$
The Riemannian metric $g$ is called bundle-like if the induced transverse metric $g^T$ is holonomy invariant, i.e.,
$L_\xi g^T=0$ for any $\xi\in \Gamma \mathcal{L}$. A foliation $(M,\mathcal{F})$ with a bundle-like metric $g$ is called a Riemannian foliation.
For the geometry of Riemannian foliations, see \cite{Tondeur}.

Let $(M,\mathcal{F},g)$ be a Riemannian foliation (i.e., $g$ is a bundle-like metric) and $\nabla^M$ be the Levi-Civita connection of $(M,g)$.
We denote the transverse Livi-Civita connection on $(\nu(\mathcal{F}), g^T)$ by $\nabla$.
This connection is defined by,
$$\nabla_X s:= \left\{
\begin{array}{lll}
\pi[X, Y_s], & \ds\vs \quad \hbox{ if } X\in \Gamma \mathcal{L} \\
\\ \pi(\nabla^M_X Y_s),& \quad \hbox{ if } X\in \Gamma\mathcal{L}^\perp,
\end{array}\right.$$
where $s\in \Gamma \nu(\mathcal{F})$ and $Y_s=\sigma s$.

 The transverse curvature operator is then defined by
$$R^T(X,Y)s=\nabla_X\nabla_Y s-\nabla_Y\nabla_X s-\nabla_{[X,Y]} s.$$
Let $\{e_i\}_{i=1}^q$ be a local orthonormal frame on $(\mathcal{L}^\perp, g|_{\mathcal{L}^\perp})$.
The transverse Ricci curvature and the transverse scalar curvature are defined respectively by
$$Ric^T(X,Y)=g^T(R^T(X,e_i)(\pi e_i),\pi Y)$$
and
$$R^T=Ric^T(e_j,e_j).$$
Both $Ric^T$ and $R^T$ are well-defined, i.e., they are independent of the choice of the local frame $\{e_i\}^q_{i=1}$.
Note that $R^T$ is basic, i.e. $\xi(R^T)=0$ for any $\xi\in \Gamma \mathcal{L}$.

We denote by $\tau$ the mean curvature vector field of leaves. That is
$$\tau=(\nabla^M_{\xi_\alpha}\xi_\alpha)^\perp,$$
where $\{\xi_\alpha\}$ is a local orthonormal frame of $\mathcal{L}$ and $X^\perp$ denotes the
projection of $X$ to $\mathcal{L}^\perp$.
A Riemannian foliation $(M,\mathcal{F},g)$ is said to be minimal if $\tau=0$.

Let $V(\mathcal{F})$ denote the space of all infinitesimal automorphisms of $\mathcal{F}$, i.e.,
$$V(\mathcal{F})=\{Y\in \Gamma(TM)|\quad L_Y \xi\in \Gamma\mathcal{L}, \quad \forall \xi \in \Gamma\mathcal{L}\}.$$
The space of transverse vector fields is defined by
$$\overline{V}(\mathcal{F})=\{\overline{Y}:=\pi (Y)| Y\in V(\mathcal{F})\}.$$
A function $f$ is called basic if $df(\xi)=0$ for any $\xi\in \Gamma \mathcal{L}$.
If a vector field $Y\in V(\mathcal{F})$ satisfies
\begin{equation}\label{eq2.1}
L_Yg^T=2f_Yg^T
\end{equation}
for a basic function $f_Y$ depending on $Y$,
we call $\overline{Y}$ (or $Y$) a transverse conformal field.
For a transverse conformal field $\overline{Y}$ we have $f_Y=\frac{1}{q} \div^{\nabla}(\overline{Y})$, where
$${\div}^\nabla(\overline{Y})=g^T(\nabla_{e_i}\overline{Y},\pi e_i)$$
is the transverse divergence of $\overline{Y}$ with respect to $\nabla$.
The following transverse divergence theorem can be found in \cite{Tondeur}. See also \cite{YorozuTanemura}.

\begin{lemma}
Let $(M,\mathcal{F},g)$ be a Riemannian foliation and $X\in V(\mathcal{F})$. Then
$$\int_M {\div}^\nabla (\overline{X})d\mu_g=\int_M g^T(\overline{X}, \pi \tau)d\mu_g.$$
\end{lemma}

The basic Laplacian $\triangle_B$ acting on a basic function $u$ is defined by $$\triangle_Bu=
{\div}^\nabla(\nabla^M u)-\tau(u).$$

\begin{lemma}
Let $(M,\mathcal{F},g)$ be a Riemannian foliation of codimension $q$.
Assume that $Y$ is a transverse conformal vector field. Then
\begin{equation}\label{ScalarTransf}
\frac{q}{2}Y(R^T)=-(q-1)(\triangle_B+\tau) {\div}^\nabla(\overline{Y})-R^T {\div}^\nabla(\overline{Y}).
\end{equation}
\end{lemma}

\begin{proof}
For the proof of formula (\ref{ScalarTransf}) see for instance \cite{JungJung}.
The classical version of (\ref{ScalarTransf}) can be found in \cite{Lichnerowicz}.
\end{proof}

Let $(M, \mathcal{F},g_0)$ be a Riemannian foliation with mean curvature vector field $\tau_0$.
Recall that any $g\in [g_0]_B$, defined in (\ref{basicconformalclass}), induces a conformal transverse metric $g^T\in [g_0^T]_B$.
Let $g=g_0|_{\mathcal{L}}+e^u g_0|_{\mathcal{L}^\perp}$ be a metric in the basic conformal class $[g_0]_B$.
Then $(M, \mathcal{F},g)$ is a Riemannian foliation with mean curvature vector field $\tau=e^{-u}\tau_0$.
Hence the condition of  minimal foliation (i.e., $\tau=0$) is conformally invariant.
A direct computation gives the relationship of $R^T_{g^T}$ and $R^T_{g_0^T}$
\begin{equation}
\label{eq2.3}
R^T_{{g^T}}=e^{-u}[-(q-1)(\triangle_B^{0}+\tau_0)u-\frac{(q-1)(q-2)}{4}|\nabla u|_{g_0^T}^2+R^T_{g_0^T}].
\end{equation}
It follows that to solve the transverse Yamabe problem is equivalent to solve the following equation
\begin{equation}\label{eq2.2}
-(q-1)(\triangle_B^{0}+\tau_0)u-\frac{(q-1)(q-2)}{4}|\nabla u|_{g_0^T}^2+R^T_{g_0^T}=ce^u.
\end{equation}
When the foliation is minimal, i.e., $\tau_0=0$, the transverse Yamabe equation (\ref{eq2.2}) is the Euler-Lagrange equation of
the following functional
\begin{equation}\label{func1}
J(g)=\frac{\int_M R^T_{g^T} d\mu_{g}} {(vol (g))^{\frac {q-2}q}}, \quad \hbox{ if } q>2
\end{equation}
and
\begin{equation}\label{func2}
J_2(g)={\int_M (-\frac 12 u\Delta_B^0 u+u R^T_{g_0^T}) d\mu_{g_0}}-\int_M R^T_{g_0^T} d\mu_{g_0}\log \int_M e^u d\mu_{g_0} , \,\hbox{ if } q=2.
\end{equation}
Recall $g=g_0|_{\mathcal{L}}+e^u g_0|_{\mathcal{L}^\perp}$. Note that, when $q=2$, from (\ref{eq2.3}) it is easy to see
$\int_M R^T_{g^T} d\mu_{g}=\int_M R^T_{g_0^T} d\mu_{g_0}$ for any $g \in [g_0]_B$.

If $\tau_0\neq 0$, equation (\ref{eq2.2}) could be not a variational problem. Therefore we have first restricted ourself to the case of minimal Riemannian foliation in our study of the transverse Yamabe problem \cite{WZ3}. From now on we consider only minimal Riemannian foliations.
Assume that $Y$ is a transverse conformal vector field. We define
\begin{equation}\label{invariantdef}
I_{\overline{Y}}: [g_0]_B\rightarrow \R; \quad I_{\overline{Y}}({g})=\int_M Y(R^T_{{g^T}})d\mu_{{g}}.
\end{equation}

\begin{theorem}\label{invariant}
Let $(M, \mathcal{F},g_0)$ be a minimal Riemannian foliation and $Y$ be a transverse conformal vector field.
Then  $I_{\overline{Y}}({g})$ is independent of the choice of $g\in [g_0]_B$.
\end{theorem}

\begin{proof}
Let $g^T$ be any transverse metric in $[g_0^T]_B$. Let $g^T(s)=e^{s\psi}g^T$, where $\psi$ is a basic function.
Let $g(s)\in [g_0]_B$ be the bundle-like metric inducing $g^T(s)$.
We denote by $R^T$ the transverse curvature of $g^T$.

It suffices to show that $\frac{d}{ds}|_{s=0}I_{\overline{Y}}(g(s))=0$.
Note that $(M, \mathcal{F},g_0)$ is a minimal Riemannian foliation, so is $(M, \mathcal{F},g)$.
Hence from (\ref{eq2.3}) we have
$$R^T_{g^T(s)}=e^{-s\psi}[-(q-1)s \triangle_B \psi-\frac{(q-1)(q-2)}{4}s^2|\nabla \psi|_{g^T}^2+R^T]$$
and
$$d\mu_{g(s)}=e^{\frac{qs\psi}{2}}d\mu_g.$$
It follows that
\begin{eqnarray*}
\frac{d}{d s}|_{s=0}I_{\overline{Y}}(g(s))&=&\int_M(Y[-\psi R^T-(q-1)\triangle_B \psi]+\frac{q}{2}\psi Y(R^T))d\mu_g
\\&=&\int_M[(\frac{q}{2}-1)\psi Y(R^T)-R^TY(\psi)-(q-1)Y \triangle_B\psi]d\mu_g.
\end{eqnarray*}
By the transverse divergence Theorem, we have
\begin{eqnarray*}
\int_M-R^TY(\psi)d\mu_g&=&-\int_M[{\div}^\nabla(\psi R^T \overline{Y})-\psi Y(R^T)-\psi R^T {\div}^\nabla (\overline{Y})]d\mu_g
\\&=&\int_M \psi[Y(R^T)+ R^T {\div}^\nabla (\overline{Y})]d\mu_g
\end{eqnarray*}
and
\begin{eqnarray*}
\int_M Y(\triangle_B\psi)d\mu_g&=&\int_M[{\div}^\nabla(\triangle_B\psi \overline{Y})-{\div}^\nabla(\overline{Y})\triangle_B\psi]d\mu_g
\\&=&-\int_M \psi\triangle_B( {\div}^\nabla \overline{Y}) d\mu_g.
\end{eqnarray*}
Hence we get
\begin{eqnarray*}
\frac{d}{d s}|_{s=0}I_{\overline{Y}}(g(s))=\int_M\psi[\frac{q}{2} Y(R^T)+R^T {\div}^\nabla (\overline{Y})+(q-1)\triangle_B( {\div}^\nabla \overline{Y})]d\mu_g.
\end{eqnarray*}
The proof is then completed by formula (\ref{ScalarTransf}).
\end{proof}

The proof given in \cite{Bourguignon} (see aslo \cite{Futaki2}) for the invariant of (\ref{int1}) would
provide another proof of this Theorem.
For foliation of higher codimension $q>2$, this invariant is still trivial.
\begin{theorem}
Let $(M,\mathcal{F},g_0)$ be a minimal Riemannian foliation of codimension $q\geq 3$
and $Y$ be a transverse conformal vector field.
Then for any $g\in [g_0]_B$, we have $I_{\overline{Y}}(g)=0$.
\end{theorem}

\begin{proof}
Let $g\in [g_0]_B$ which induces $g^T$.
Integrating (\ref{ScalarTransf}) and using the transverse divergence Theorem, we get
$$(\frac{q}{2}-1)\int_M Y(R^T_{g^T})d\mu_g=0.$$
Hence, $I=0$. \end{proof}

\section{Examples with non-vanishing $I_{Y}$}

In this section we compute the invariant, defined in (\ref{invariantdef}),
for a family of codimension $2$ minimal Riemannian foliations.
This family is given by the weighted Sasakian structures on the unit sphere in $\C^2$.
It is well-known in Sasakian geometry that a Sasakian manifold admits a minimal Riemannian foliation structure.
For the weighted Sasakian structures on the unit sphere in $\C^2$ one may refer to \cite{BoyerGalicki,Gauduchon,Wang}.

On the unit sphere $\S^3$ in $\C^2$ there is a canonical Sasakian structure described as below.
Let $$\eta=\sum_{i=1}^2(x^idy^i-y^idx^i)$$ be the canonical contact form on $\S^3$.
It uniquely determines a vector field $\xi$ by $\eta(\xi)=1$ and $i_\xi d\eta=0$.
The vector field $\xi$ is called the Reeb vector field and
$$\xi=\sum_{i=1}^2(x^i\frac{\partial}{\partial y^i}-y^i\frac{\partial}{\partial x^i}).$$
The Reeb vector field $\xi$ gives rise to a foliation $\mathcal{F}_\xi$ on $\S^3$.
The distribution $\mathcal{D}:=\text{ker}\, \eta$ is called the contact distribution.
Let $\phi$ be the linear map which satisfies $\phi\xi=0$
and $\phi|_{\mathcal{D}}$ is the restriction of the canonical complex structure of $\C^2$.
Let
\begin{equation}\label{g}
g=d\eta\circ(Id\otimes \phi)+\eta\otimes\eta.
\end{equation}
Then $g$ is a bundle-like metric for the foliation $\mathcal{F}_\xi$ and $\mathcal{D}=\mathcal{L}_\xi^\perp$,
where $\mathcal{L}_\xi$ is the integrable line bundle given by $\mathcal{F}_\xi$.
In fact $g$ is the standard metric on the unit sphere and $R^T=8$.

For a given pair $(a_1, a_2)$ of positive numbers, the weighted Sasakian structure $\S_a^3$
on the unit sphere $\S^3$ is defined as follows.
Let $$\eta_a=\sigma^{-1}\sum_{i=1}^2(x^idy^i-y^idx^i)=\sigma^{-1}\eta,$$ where $\sigma=a_1|z_1|^2+a_2|z_2|^2$.
The contact form $\eta_a$ uniquely determines the Reeb vector field $\xi_a$ by
$\eta_a(\xi_a)=1$ and $i_{\xi_a} d\eta_a=0$.
It is trivial to see that $\text{ker}\,\eta_a=\text{ker}\,\eta=\mathcal{D}$ and easy to check
$$\xi_a=\sum_{i=1}^2a_i(x^i\frac{\partial}{\partial y^i}-y^i\frac{\partial}{\partial x^i}).$$
The line bundle $\mathcal{L}_{\xi_a}$, spanned by the Reeb vector field $\xi_a$,
defines the characteristic foliation structure $\mathcal{F}_{\xi_a}$ of $\S^3_a$.
Let $\phi_a$ be the linear map which satisfies $\phi_a\xi_a=0$ and $\phi_a|_{\mathcal{D}}=\phi|_{\mathcal{D}}$;
and let
\begin{equation}\label{g2}
g_a=d\eta_a\circ(Id\otimes \phi_a)+\eta_a\otimes\eta_a.\end{equation}
Then $(\S^3,\mathcal{F}_{\xi_a},g_a)$ is a Sasakian manifold,
and hence its foliation is minimal. See \cite{Gauduchon}.

If $\frac{a_1}{a_2}(\neq 1)$ is a rational number, the leaf space of $\mathcal{F}_{\xi_a}$ is
the orbifold ${\mathbb P}_\C^1(a_1,a_2)$, i.e., the weighted projective line.
Any holonomy invariant transverse metric $g^T$ on $(\S^3,\mathcal{F}_{\xi_a})$
induces naturally an orbifold metric on $\mathbb{P}_{\C}^1(a_1,a_2)$
and the scalar curvature of the orbifold metric is exactly the transverse scalar curvature $R^T_{g^T}$.
It was known that $\mathbb{P}_{\C}^1(a_1,a_2)$ admits no orbifold metric of constant scalar curvature,
see \cite{Wu}. In \cite{GauduchonJDG} the Futaki invariant of $\mathbb{P}_{\C}^1(a_1,a_2)$ was shown to be non-vanishing.

If $\frac{a_1}{a_2}$ is an irrational number, the leaf space of $\mathcal{F}_{\xi_a}$ has no manifold structure.
In the following we show that the invariant $I$ is zero if and only if $a_1=a_2$, i.e., the leaf space is the
standard sphere. This phenomenon is very similar to the existence of $\S^1$-equivariant harmonic maps from
$\S^3$ into $\S^2$ in \cite{Wang0}.


We have the following two (real) tangent vector fields on $\S^3$:
$$Z_1=\sigma^{-1}(-i|z_2|^2z_1,i|z_1|^2z_2), \quad Z_2=\sigma^{-1}(|z_2|^2z_1,-|z_1|^2z_2),$$
where $z_1$ denotes the vector field $x^1\frac{\partial}{\partial x^1}+y^1\frac{\partial}{\partial y^1}$
and $iz_1$ the vector field $-y^1\frac{\partial}{ \partial x^1}+x^1\frac{\partial}{\partial y^1}$.
A direct computation gives
$$Z_i\in {\rm ker}\,\eta_a, \quad Z_2=\phi_a Z_1, \quad [\xi_a,Z_i]=0, \quad [Z_1,Z_2]=-2\sigma^{-3}|z_1|^2|z_2|^2\xi_a.$$

\begin{proposition}
The vector field $Z_2$ is a transverse conformal vector field on $(\S^3,\mathcal{F}_{\xi_a},g_a)$.
\end{proposition}

\begin{proof}
By the fact that $[\xi_a,Z_2]=0$, we see that $Z_2\in V(\mathcal{F}_{\xi_a})$.
Note that $g$, given by (\ref{g}), is the standard metric on $\S^3$ and $Z_1, Z_2\in \mathcal{D}$.
Then by the definition (\ref{g2}) of $g_a$, we have
$$g_a^T(Z_i,Z_j)=\sigma^{-1}g^T(Z_i,Z_j)=\sigma^{-1}g(Z_i,Z_j)=\sigma^{-3}|z_1|^2|z_2|^2\delta_{ij}.$$
One can then verify that
$L_{Z_2}g_a^T=2f_{Z_2}g_a^T$ with
$$ f_{Z_2}=\frac 12 Z_2\log (\sigma^{-3}|z_1|^2|z_2|^2).$$
\end{proof}

\begin{proposition} \label{transversescalar} The transverse scalar curvature of
$(S^3,\mathcal{F}_{\xi_a},g_a)$ is
\begin{eqnarray*}
R^T(g_a^T)=-24(a_1-a_2)^2\sigma^{-1}|z_1|^2|z_2|^2-16(a_1-a_2)(|z_1|^2-|z_2|^2)+8\sigma.
\end{eqnarray*}
\end{proposition}
\begin{proof}
Let $Z=Z_1-iZ_2$. We have $\phi Z=iZ$ and $g_a^T(Z,\overline{Z})=2\sigma^{-3} |z_1|^2|z_2|^2$,
where $\overline{Z}$ denotes the complex conjugate of $Z$.
The transversal scalar curvature follows from the following formula
$$R^T(g_a^T)=-2[g_a^T(Z,\overline{Z})]^{-1}Z\overline{Z}\log [g_a^T(Z,\overline{Z})].$$
We notice that the transverse scalar curvature of $\S_a^3$ was also given in \cite{Gauduchon} in a slightly different form.
\end{proof}

\begin{proposition} \label{invariantonweighted}
On the minimal Riemannian foliation $(\S^3,\mathcal{F}_{\xi_a},g_a)$, we have
$$I_{Z_2}(\widetilde{g})= -8\pi^2\frac{a_1^2-a_2^2}{a_1^2a_2^2}, \quad \forall \widetilde{g}\in [g_a]_B.$$
\end{proposition}

\begin{proof}
By Theorem \ref{invariant}, it suffices to show that
$$\int_{S^3}Z_2(R^T_{g_a^T})d\mu_{g_a}=-8\pi^2\frac{a_1^2-a_2^2}{a_1^2a_2^2}.$$
It follows from Proposition \ref{transversescalar} that
$$Z_2(R^T_{g_a^T})=-48(a_1-a_2)a_1a_2\sigma^{-3}|z_1|^2|z_2|^2.$$
Let $t=|z_1|^2, s=(a_1-a_2)t+a_2=\sigma$. Then
$$Z_2(R^T_{g_a^T})=48 (a_1-a_2)^{-1} a_1 a_2 s^{-3}(s-a_1)(s-a_2).$$
Note that the volume element with respect to $g_a$ is
$$d\mu_{g_a}=\eta_a\wedge d\eta_a=\sigma^{-2}\eta\wedge d\eta=\sigma^{-2}d\mu_g,$$
where $d\mu_g$ is the standard volume element of $\S^3$.
For each $t=|z_1|^2\in (0,1)$, we have a torus $\{(z_1,z_2): |z_1|^2=t, |z_2|^2=1-t\}$ which has area
$(2\pi\sqrt{t})(2\pi\sqrt{1-t})$ with respect to the standard metric $g$ on $\S^3$.
Note that
$$\xi, \quad (|z_1| |z_2|)^{-1}\sigma Z_1, \quad(|z_1| |z_2|)^{-1}\sigma Z_2$$
form an orthonormal frame of $(\S^3,g)$. It is easy to see that $\xi|z_1|^2=Z_1|z_1|^2=0$.
It implies that $$|\nabla |z_1|^2|_g=(|z_1| |z_2|)^{-1} \sigma |Z_2(|z_1|^2)|=2|z_1| |z_2|.$$
Then by the coarea formula, we have
\begin{eqnarray*}\int_{\S^3} Z_2(R^T_{g_a^T})\eta_a\wedge d\eta_a
&=&\int_0^1 Z_2(R^T_{g_a^T}) \sigma^{-2}(2\pi\sqrt{t})(2\pi\sqrt{1-t})\frac{dt}{|\nabla |z_1|^2|_{g}}
\\&=&\int_0^1 Z_2(R^T_{g_a^T}) \sigma^{-2}2\pi^2dt
\\&=&\int_{a_2}^{a_1}96\pi^2(a_1-a_2)^{-2}a_1a_2 s^{-5}(s-a_1)(s-a_2)d s
\\&=&-8\pi^2\frac{a_1^2-a_2^2}{a_1^2a_2^2}.
\end{eqnarray*}
\end{proof}

Therefore the invariant vanishes if and only if $a_1=a_2$. This also gives the proof of Theorem \ref{mainthm2}.

\end{document}